\documentclass[12pt]{article}
\usepackage{amssymb}
\usepackage{amsfonts}
\usepackage{amsmath}
 \usepackage{color}
\usepackage{mathrsfs}
\usepackage{amsfonts}
\usepackage{amssymb,amsmath}
\usepackage{CJK}
\usepackage{cite}
\usepackage{cases}
\usepackage{amsthm}

\def\cl{\centerline}

\def\vs{\vspace*}

\numberwithin{equation}{section}
\newtheorem{theo}{Theorem}[section]
\newtheorem{defi}[theo]{Definition}
\newtheorem{coro}[theo]{Corollary}
\newtheorem{lemm}[theo]{Lemma}
\newtheorem{exam}[theo]{Example}
\newtheorem{prop}[theo]{Proposition}

\newtheorem{rema}[theo]{Remark}

\begin{document}
\begin{center}
{\bf\large Extending structures for Lie bialgebras}
\end{center}

\cl{Yanyong Hong
}

\vs{8pt}

{\small\footnotesize
\parskip .005 truein
\baselineskip 3pt \lineskip 3pt
\noindent{{\bf Abstract: Let $(\mathfrak{g}, [\cdot,\cdot], \delta_\mathfrak{g})$ be a fixed Lie bialgebra, $E$ be a vector space containing $\mathfrak{g}$ as a subspace and $V$ be a complement of $\mathfrak{g}$ in $E$. A natural problem is that how to classify all Lie bialgebraic structures on $E$ such that $(\mathfrak{g}, [\cdot,\cdot], \delta_\mathfrak{g})$ is a Lie sub-bialgebra up to an isomorphism of Lie bialgebras whose restriction on $\mathfrak{g}$ is the identity map. This problem is called the extending structures problem. In this paper, we introduce a general co-product on $E$, called the unified co-product of
$(\mathfrak{g},\delta_\mathfrak{g})$ by $V$. With this unified co-product and the unified product of $(\mathfrak{g}, [\cdot,\cdot])$ by $V$ developed in \cite{AM1}, the unified bi-product of $(\mathfrak{g}, [\cdot,\cdot], \delta_\mathfrak{g})$ by $V$ is introduced. Moreover, we show that any $E$ in the extending structures problem is isomorphic to a unified bi-product of $(\mathfrak{g}, [\cdot,\cdot], \delta_\mathfrak{g})$ by $V$. Then an object $\mathcal{HBI}_{\mathfrak{g}}^2(V,\mathfrak{g})$ is constructed to classify all $E$ in the extending structures problem.
Moreover, several special unified bi-products are also introduced. In particular, the unified bi-products when $\text{dim} V=1$ are investigated in detail.
  \vs{5pt}

\noindent{\bf Key words:} Lie bialgebra, Extending structure
\parskip .001 truein\baselineskip 6pt \lineskip 6pt

\noindent{\it Mathematics Subject Classification (2010):} 17A30,  17B62, 17B65, 17B69.}}
\parskip .001 truein\baselineskip 6pt \lineskip 6pt

\section{Introduction}
Lie bialgebras were first introduced by Drinfel'd in \cite{Dr1} as the algebraic structures
underlying quantized enveloping algebras (quantum groups). A Lie bialgebra $(\mathfrak{g}, [\cdot,\cdot], \delta_\mathfrak{g})$ is a Lie algebra $(\mathfrak{g}, [\cdot,\cdot])$ with a Lie co-bracket $\delta_\mathfrak{g}: \mathfrak{g}\rightarrow \mathfrak{g}\otimes \mathfrak{g}$ satisfying some compatibility conditions.  Since the Lie algebra of a Poisson-Lie group is a Lie bialgebra, the theory of Lie bialgebra is also closely related with the theory of Poisson-Lie group.  Moreover, Lie bialgebras and Poisson-Lie groups have applications in the theory of Poisson-Lie $T$-dual sigma models (see \cite{KS, K}). Therefore, the classification of Lie bialgebras is important in the theories of quantum groups and Poisson Lie groups. In the aspect of classifications, there have been many works. For example, many interesting examples of Lie bialgebras based on complex semisimple Lie algebras were presented by Drinfel'd in \cite{Dr1}, Lie bialgebras of low dimensions over different fields were classified in \cite{ARK, FJ1, G, J, RHR, Z1, Z2  } respectively, and Lie bialgebras on infinite-dimensional Lie algebras such as Witt algebra and Virasoro algebra were classified in \cite{T, NT} and so on. These works are based on the same idea which is to classify all Lie bialgebraic structures on a given Lie algebra.

Another method to construct Lie bialgebraic structures is through extensions of a Lie bialgebra by another Lie bialgebra. Such extensions were investigated in \cite{B, B1, FJ, Ma} and so on.
In this paper, we investigate a more general problem as follows.\\
{\bf Extending structures problem}: Let $(\mathfrak{g}, [\cdot, \cdot], \delta_{\mathfrak{g}})$ be a Lie bialgebra and $E$ a vector space containing $\mathfrak{g}$ as a subspace. Describe and classify all Lie bialgebraic structures on $E$ such that $(\mathfrak{g}, [\cdot, \cdot], \delta_{\mathfrak{g}})$  is a Lie sub-bialgebra of $(E, [\cdot, \cdot], \delta_{E})$ up to an isomorphism of Lie bialgebras whose restriction on $\mathfrak{g}$ is the identity map.

Similar problems for Hopf algebras and Lie algebras were also investigated in \cite{AM2} and \cite{AM1} respectively. It is a natural and important problem from the point of view of algebraic theory, i.e. how to obtain a `big' Lie bialgebra from a given `small' Lie bialgebra. Motivated by the theory of extending structures for Lie algebras in \cite{AM1}, we first introduce a definition of unified co-product for Lie coalgebras. By means of unified co-product, we show that any Lie coalgebra structure on $E$ such that $(\mathfrak{g},\delta_{\mathfrak{g}})$ is a Lie sub-coalgebra is isomorphic to a unified co-product of $(\mathfrak{g},\delta_{\mathfrak{g}})$ by $V$ which is a complement of $\mathfrak{g}$ in $E$. Then an object $\mathcal{HC}_\mathfrak{g}^2(V,\mathfrak{g})$ is constructed to  classify extending structures for Lie coalgebras.  Note that a Lie bialgebra is both a Lie algebra and a Lie coalgebra. With the theories of extending structures for Lie algebras and Lie coalgebras, a unified bi-product for Lie bialgebras is naturally presented and we also construct an object $\mathcal{HBI}_\mathfrak{g}^2(V,\mathfrak{g})$ to give a characterization of extending structures for Lie bialgebras. Moreover, several special unified bi-products for Lie bialgebras are also introduced. In addition, we investigate the unified bi-products when $\text{dim}V=1$ in detail. It is shown that $\mathcal{HBI}_\mathfrak{g}^2(V,\mathfrak{g})$ can be computed in this case.

This paper is organized as follows. In Section 2, we recall some preliminaries about Lie bialgebras and unified product for Lie algebras. In Section 3, we introduce  a definition of unified co-product for Lie coalgebras and construct an object $\mathcal{HC}_\mathfrak{g}^2(V,\mathfrak{g})$ to characterize extending structures for Lie coalgebras. In Section 4, a general theory of extending structures for Lie bialgebras is given. Section 5 is devoted to introducing some special unified bi-products for Lie bialgebras, such as crossed bi-product and double cross sum. Finally, in Section 6, the unified bi-products when $\text{dim}V=1$ are investigated in detail and an example to compute $\mathcal{HBI}_\mathfrak{g}^2(V,\mathfrak{g})$ is presented.

Throughout this paper, $k$ is a field with $\text{Char} k=0$ and $\mathbb{C}$ is the field of complex numbers. Set $k^\ast=k\backslash\{0\}$. All vector spaces, Lie algebras, Lie coalgebras, Lie bialgebras, linear or bilinear maps are over $k$. $I$ usually represents the identity map.

\section{Preliminaries on Lie bialgebras}
In this section, we recall some definitions about Lie bialgebras and some facts about unified product for Lie algebras.

Recall that a {\bf Lie algebra} $\mathfrak{g}$ is a vector space endowed with a linear map $[\cdot, \cdot]: \mathfrak{g}\otimes \mathfrak{g}\rightarrow \mathfrak{g}$ satisfying
\begin{eqnarray*}
&&[a, a]=0,\\
&&[a,[b,c]]=[[a,b],c]+[b,[a,c]],\;\;\;\text{(Jacobi identity)}
\end{eqnarray*}
for all $a, b, c\in \mathfrak{g}$.
Let $\mathfrak{g}$ be a Lie algebra. A {\bf right $\mathfrak{g}$-module} is a vector space $V$ together with a linear map $\lhd: V\otimes \mathfrak{g}\rightarrow V$ satisfying
\begin{eqnarray*}
x\lhd [a, b]=(x\lhd a)\lhd b-(x\lhd b)\lhd a,\;\;x\in V, a, b\in \mathfrak{g}.
\end{eqnarray*}

A {\bf Lie coalgebra} \cite{Mi} is a vector space $\mathfrak{g}$ equipped with a linear map $\delta_\mathfrak{g}:\mathfrak{g}\rightarrow \mathfrak{g}\otimes \mathfrak{g}$ satisfying
\begin{eqnarray*}
&&\delta_\mathfrak{g}=-\tau \delta_\mathfrak{g},\\
&&(I\otimes \delta_\mathfrak{g})\delta_\mathfrak{g}-\tau_{12}(I\otimes \delta_\mathfrak{g})\delta_\mathfrak{g}=(\delta_\mathfrak{g}\otimes I)\delta_\mathfrak{g},\;\;\; \text{(co-Jacobi identity)}
\end{eqnarray*}
where $\tau: \mathfrak{g}\otimes \mathfrak{g}\rightarrow \mathfrak{g}\otimes \mathfrak{g}$,
$\tau_{12}:\mathfrak{g}\otimes \mathfrak{g}\otimes \mathfrak{g}\rightarrow \mathfrak{g}\otimes \mathfrak{g}\otimes \mathfrak{g}$ are the twistings defined by
\begin{eqnarray*}
\tau (a\otimes b)=b\otimes a, ~~~~\tau_{12}(a\otimes b\otimes c)=b\otimes a \otimes c, \;\;(\text{$a$, $b$, $c\in \mathfrak{g}$}).
\end{eqnarray*}
We denote it by $(\mathfrak{g}, \delta_\mathfrak{g})$.

A {\bf Lie bialgebra} \cite{CP} is a triple $(\mathfrak{g}, [\cdot,\cdot],\delta_\mathfrak{g})$ such that $(\mathfrak{g}, [\cdot, \cdot])$ is a Lie algebra,
$(\mathfrak{g}, \delta_\mathfrak{g})$ is a Lie coalgebra, and they satisfy the cocycle condition
\begin{eqnarray}
\label{lia3}a.\delta_\mathfrak{g}(b)-b.\delta_\mathfrak{g}(a)=\delta_\mathfrak{g}([a,b]),\;\;\forall a, b\in \mathfrak{g},
\end{eqnarray}
where $a. \delta_\mathfrak{g}(b)=\sum ( [a, b_{(1)}]\otimes b_{(2)}+b_{(1)}\otimes [a, b_{(2)}])$, if $\delta_\mathfrak{g}(b)=\sum b_{(1)}\otimes b_{(2)}$.

A {\bf homomorphism of Lie bialgebras} $\varphi: (\mathfrak{g}_1, [\cdot,\cdot],\delta_{\mathfrak{g}_1})\rightarrow (\mathfrak{g}_2, [\cdot,\cdot],\delta_{\mathfrak{g}_2})$ is both a homomorphism of Lie algebras and a homomorphism of Lie coalgebras, i.e. for all $a$, $b\in \mathfrak{g}_1$,
\begin{eqnarray*}
\varphi([a, b])=[\varphi(a), \varphi(b)],\;\; \delta_{\mathfrak{g}_2}\varphi(a)=(\varphi\otimes \varphi) \delta_{\mathfrak{g}_1}(a).
\end{eqnarray*}

In order to investigate extending structures problem for Lie bialgebras, we introduce the following definition.
\begin{defi}
Let $(\mathfrak{g}, [\cdot,\cdot],\delta_\mathfrak{g})$ be a Lie bialgebra, $E$ a vector space containing $\mathfrak{g}$.

Let $(E, \{\cdot,\cdot\}, \delta_E)$ and $(E,\{\cdot,\cdot\}^{'}, \delta_E^{'})$ (resp.  $(E, \delta_E)$ and $(E,\delta_E^{'})$) be two Lie bialgebras (resp. Lie coalgebras) containing $(\mathfrak{g}, [\cdot,\cdot],\delta_\mathfrak{g})$ (resp. $(\mathfrak{g}, \delta_\mathfrak{g})$) as a Lie sub-bialgebra (resp. Lie sub-coalgebra). If there is an isomorphism of Lie bialgebras $\varphi:(E, \{\cdot,\cdot\}, \delta_E)\rightarrow (E,\{\cdot,\cdot\}^{'}, \delta_E^{'})$ (resp. Lie coalgebras $\varphi:(E, \delta_E)\rightarrow (E,\delta_E^{'})$) whose restriction on $\mathfrak{g}$ is the identity map, then we call that $(E, \{\cdot,\cdot\}, \delta_E)$ and $(E,\{\cdot,\cdot\}^{'}, \delta_E^{'})$ (resp. $(E, \delta_E)$ and $(E,\delta_E^{'})$) are {\bf equivalent}. Denote it by $(E, \{\cdot,\cdot\}, \delta_E)\equiv (E,\{\cdot,\cdot\}^{'}, \delta_E^{'})$ (resp. $(E, \delta_E)\equiv (E, \delta_E^{'})$).
\end{defi}

Obviously, $\equiv$ is an equivalence relation on the set of all Lie bialgebraic (resp. Lie coalgebraic) structures on $E$ containing $(\mathfrak{g}, [\cdot,\cdot],\delta_\mathfrak{g})$ (resp. $(\mathfrak{g},\delta_\mathfrak{g})$) as a Lie sub-bialgebra (resp. Lie sub-coalgebra). We use $BExtd(E,\mathfrak{g})$ (resp. $CExtd(E,\mathfrak{g})$) to denote the set of all equivalence classes via $\equiv$. Therefore, for answering the extending structures problem for Lie bialgebras, we only need to characterize $BExtd(E,\mathfrak{g})$. Since a Lie bialgebra is both a Lie algebra and a Lie coalgebra, we need to know the theory of extending structures for Lie algebras and investigate $CExtd(E,\mathfrak{g})$.

Finally, we recall some facts about extending structures for Lie algebras developed in \cite{AM1}.
\begin{defi}
Let $\mathfrak{g}$ be a Lie algebra and $V$ a vector space. An {\bf extending datum} of $\mathfrak{g}$ by $V$ is a system $\Omega(\mathfrak{g},V)=(\lhd, \rhd, f, \{\cdot,\cdot\})$ consisting of four bilinear maps
\begin{eqnarray*}
\lhd: V\times \mathfrak{g}\rightarrow V,~~~~\rhd: V\times \mathfrak{g}\rightarrow \mathfrak{g},~~~~f:  V\times V \rightarrow \mathfrak{g},~~~ \{\cdot,\cdot\}:V\times V \rightarrow V.
\end{eqnarray*}
Let $\Omega(\mathfrak{g}, V)$ be an extending datum of $\mathfrak{g}$ by $V$. Denote by $\mathfrak{g}\natural_{\Omega(\mathfrak{g},V)}V=\mathfrak{g}\natural V$ the vector space $E=\mathfrak{g}\oplus V$ with the bilinear map
$[\cdot,\cdot]: E\times E \rightarrow E$ given by
\begin{eqnarray}\label{qL1}
[a+x, b+y]:=[a,b]+x\rhd b-y\rhd a+f(x,y)+x\lhd b-y\lhd a+\{x,y\},
\end{eqnarray}
for all $a$, $b\in \mathfrak{g}$, $x$, $y\in V$.
If $\mathfrak{g}\natural V$ is a Lie algebra with the Lie bracket given by (\ref{qL1}), then $\mathfrak{g}\natural V$ is called
the {\bf unified product} of $\mathfrak{g}$ and $\Omega(\mathfrak{g},V)$, and $\Omega(\mathfrak{g},V)$ is called a {\bf Lie extending structure} of $\mathfrak{g}$ by $V$.
\end{defi}

It was proved in Theorem 2.2 in \cite{AM1} that $\Omega(\mathfrak{g},V)$ is a Lie extending structure of $\mathfrak{g}$ by $V$ if and only if the following compatibility conditions hold for all $g$, $h\in \mathfrak{g}$, $x$, $y$, $z\in V$:
\begin{eqnarray*}
(LE1)~~&&f(x,x)=0,~~~\{x,x\}=0,\\
(LE2)~~&& \text{$(V,\lhd)$ is a right $\mathfrak{g}$-module},\\
(LE3)~~&&x\rhd [g,h]=[x\rhd g,h]+[g, x\rhd h]+(x\lhd g)\rhd h-(x\lhd h)\rhd g,\\
(LE4)~~&&\{x,y\}\lhd g=\{x,y\lhd g\}+\{x\lhd g,y\}+x\lhd(y\rhd g)-y\lhd(x\rhd g),\\
(LE5)~~&&\{x,y\}\rhd g=x\rhd(y\rhd g)-y\rhd(x\rhd g)+[g,f(x,y)]+f(x,y\lhd g)+f(x\lhd g,y),\\
(LE6)~~&& f(x,\{y,z\})+f(y,\{z,x\})+f(z,\{x,y\})+x\rhd f(y,z)+y\rhd f(z,x)+z\rhd f(x,y)=0,\\
(LE7)~~&&\{x,\{y,z\}\}+\{y,\{z,x\}\}+\{z,\{x,y\}\}+x\lhd f(y,z)+y\lhd f(z,x)+z\lhd f(x,y)=0.
\end{eqnarray*}

\section{Unified co-products for Lie coalgebras}
In this section, we introduce a definition of unified co-product associated with a Lie coalgebra $(\mathfrak{g},\delta_\mathfrak{g})$ and a vector space $V$. Using this definition, we construct an object $\mathcal{HC}_{\mathfrak{g}}^2(V,\mathfrak{g})$ to describe and classify all Lie coalgebraic structures on the vector space $E=\mathfrak{g}\oplus V$ such that $(\mathfrak{g},\delta_\mathfrak{g})$ is a Lie sub-coalgebra of $E$ up to an isomorphism of Lie coalgebras whose restriction on $\mathfrak{g}$ is the identity map.

\begin{defi}
Let $(\mathfrak{g},\delta_\mathfrak{g})$ be a Lie coalgebra and $V$ a vector space. An {\bf extending datum} of $(\mathfrak{g},\delta_\mathfrak{g})$ by $V$ is a system $\Omega^c(\mathfrak{g},\delta_\mathfrak{g},V)=(\Delta_E, \Delta_V, \delta_V)$ with three linear maps
\begin{eqnarray*}
\Delta_E: V\rightarrow \mathfrak{g}\otimes V,~~~~\Delta_V: V\rightarrow \mathfrak{g}\otimes \mathfrak{g},~~~~\delta_V: V\rightarrow V\otimes V.
\end{eqnarray*}
Suppose that $\Omega^c(\mathfrak{g},\delta_\mathfrak{g},V)$ is an extending datum of $(\mathfrak{g},\delta_\mathfrak{g})$ by $V$. Denote by $\mathfrak{g}\natural_{\Omega^c(\mathfrak{g},\delta_\mathfrak{g},V)}V=\mathfrak{g}\natural^cV$ the vector space $E=\mathfrak{g}\oplus V$ with the linear map
$\delta_E: E\rightarrow E\otimes E$ given by
\begin{eqnarray}\label{eq1}
\delta_E(a+x)=\delta_\mathfrak{g}(a)+\Delta_E(x)-\tau\Delta_E(x)+\Delta_V(x)+\delta_V(x),~~~\text{for all $a\in \mathfrak{g}$, $x\in V$.}
\end{eqnarray}
If $\mathfrak{g}\natural^cV$ is a Lie coalgebra with the Lie co-bracket given by (\ref{eq1}), then $\mathfrak{g}\natural^cV$ is called
the {\bf unified co-product} of $(\mathfrak{g},\delta_{\mathfrak{g}})$ and $\Omega^c(\mathfrak{g},\delta_{\mathfrak{g}},V)$, and $\Omega^c(\mathfrak{g},\delta_\mathfrak{g},V)$ is called a {\bf Lie coalgebraic extending structure} of $(\mathfrak{g},\delta_{\mathfrak{g}})$ by $V$.
Denote the set of all Lie coalgebraic extending structures of $(\mathfrak{g},\delta_{\mathfrak{g}})$ by $V$ by $\mathcal{LC}(\mathfrak{g},V)$.
\end{defi}

\begin{theo}\label{t1}
Let $(\mathfrak{g},\delta_{\mathfrak{g}})$ be a Lie coalgebra, $V$  a vector space and $\Omega^c(\mathfrak{g},\delta_{\mathfrak{g}},V)$ be an extending datum of
$(\mathfrak{g},\delta_{\mathfrak{g}})$ by $V$. Then $\mathfrak{g}\natural^cV$ is a unified co-product if and only if the following
compatibility conditions hold for all $x\in V$:
\begin{eqnarray*}
(CLE1)&&~~\Delta_V(x)=-\tau\Delta_V(x),~~~\delta_V(x)=-\tau\delta_V(x),\\
(CLE2)&&~~(I\otimes\Delta_V)\Delta_E(x)+(I\otimes \delta_\mathfrak{g})\Delta_V(x)-\tau_{12}(I\otimes \Delta_V)\Delta_E(x)-\tau_{12}(I\otimes \delta_\mathfrak{g})\Delta_V(x)\\
&&=-(\Delta_V\otimes I)\tau\Delta_E(x)
+(\delta_\mathfrak{g}\otimes I)\Delta_V(x),\\
(CLE3)&&~~(I\otimes \Delta_E)\Delta_E(x)-\tau_{12}(I\otimes \Delta_E)\Delta_E(x)=(\delta_\mathfrak{g}\otimes I)\Delta_E(x)+(\Delta_V\otimes I)\delta_V(x),\\
(CLE4)&&~~(I\otimes \delta_V)\Delta_E(x)-\tau_{12}(I\otimes \Delta_E)\delta_V(x)=(\Delta_E\otimes I)\delta_V(x),\\
(CLE5)&&~~(I\otimes \delta_V)\delta_V(x)-\tau_{12}(I\otimes \delta_V)\delta_V(x)=(\delta_V\otimes I)\delta_V(x).
\end{eqnarray*}

\end{theo}
\begin{proof}
By (\ref{eq1}), it is easy to see that $\delta_E(x)=-\tau\delta_E(x)$ holds if and only if $(CLE1)$ holds.
Then we only need to prove that co-Jacobi identity for $\delta_E$ holds for all $x\in V$ if and only if $(CLE2)$-$(CLE5)$ hold when $(CLE1)$ is satisfied.

According to (\ref{eq1}), $(I\otimes \delta_E)\delta_E(x)-\tau_{12}(I\otimes \delta_E)\delta_E(x)=
(\delta_E\otimes I)\delta_E(x)$ for all $x\in V$ if and only if
\begin{eqnarray*}
&&(I\otimes \Delta_E)\Delta_E(x)-(I\otimes \tau\Delta_E)\Delta_E(x)
+(I\otimes\Delta_V)\Delta_E(x)+(I\otimes \delta_V)\Delta_E(x)-(I\otimes\delta_\mathfrak{g})\tau\Delta_E(x)\\
&&+(I\otimes \delta_\mathfrak{g})\Delta_V(x)+(I\otimes\Delta_E)\delta_V(x)-(I\otimes \tau\Delta_E)\delta_V(x)
+(I\otimes \Delta_V)\delta_V(x)+(I\otimes \delta_V)\delta_V(x)\\
&&-\tau_{12}(I\otimes\Delta_E)\Delta_E(x)+\tau_{12}(I\otimes \tau\Delta_E)\Delta_E(x)-\tau_{12}(I\otimes\Delta_V)\Delta_E(x)-\tau_{12}(I\otimes \delta_V)\Delta_E(x)\\
&&+\tau_{12}(I\otimes \delta_\mathfrak{g})\tau\Delta_E(x)-\tau_{12}(I\otimes\delta_\mathfrak{g})\Delta_V(x)-\tau_{12}(I\otimes \Delta_E)\delta_V(x)+\tau_{12}(I\otimes \tau\Delta_E)\delta_V(x)\\
&&-\tau_{12}(I\otimes\Delta_V)\delta_V(x)-\tau_{12}(I\otimes\delta_V)\delta_V(x)\\
&&=(\delta_\mathfrak{g}\otimes I)\Delta_E(x)-(\Delta_E\otimes I)\tau\Delta_E(x)
+(\tau\Delta_E\otimes I)\tau\Delta_E(x)-(\Delta_V\otimes I)\tau\Delta_E(x)-(\delta_V\otimes I)\tau\Delta_E(x)\\
&&+(\delta_\mathfrak{g}\otimes I)\Delta_V(x)+(\Delta_E\otimes I)\delta_V(x)-(\tau\Delta_E\otimes I)\delta_V(x)+(\Delta_V\otimes I)\delta_V(x)+(\delta_V\otimes I)\delta_V(x).
\end{eqnarray*}
Since different terms in the equality above may be in different vector spaces, we put all terms in the same vector space together. Then it is easy to see that co-Jacobi identity holds if and only if
$(CLE2)$-$(CLE5)$ hold by $(CLE1)$.
\end{proof}
\begin{rema}
Note that $(V,\delta_V)$ is a Lie coalgebra by $(CLE1)$ and $(CLE5)$.
\end{rema}

Obviously, $\mathfrak{g}\natural^cV$ is a Lie coalgebra containing $(\mathfrak{g},\delta_\mathfrak{g})$ as a Lie sub-coalgebra. Next, we show that any Lie coalgebraic structure on $E$ containing $(\mathfrak{g},\delta_\mathfrak{g})$ as a Lie sub-coalgebra is isomorphic to such a unified co-product.
\begin{theo}\label{t2}
Let $(\mathfrak{g},\delta_\mathfrak{g})$ be a Lie coalgebra and $E$ a vector space containing $\mathfrak{g}$ as a subspace. Suppose that there is a Lie coalgebraic structure $(E,\delta_E)$ on $E$ such that  $(\mathfrak{g},\delta_\mathfrak{g})$ is a Lie sub-coalgebra of $E$. Then there exists a Lie coalgebraic extending structure $\Omega^c(\mathfrak{g},\delta_\mathfrak{g}, V)$ of $(\mathfrak{g},\delta_\mathfrak{g})$ by $V$ such that $(E,\delta)\cong \mathfrak{g}\natural^c V$ whose restriction on $\mathfrak{g}$ is the identity map.
\end{theo}
\begin{proof}
Suppose that the natural projection from $E$ to $\mathfrak{g}$ is $\pi_1$ and $V=ker(\pi_1)$. Denote the natural projection from $E$ to $V$ by $\pi_2$. We define the following extending datum of $(\mathfrak{g},\delta_\mathfrak{g})$ by $V$ as follows:\\
\begin{eqnarray*}
&&\Delta_V: V\rightarrow \mathfrak{g}\otimes \mathfrak{g},~~~~\Delta_V(x)=(\pi_1\otimes \pi_1)\delta_E(x),\\
&&\delta_V: V\rightarrow V\otimes V,~~~~\delta_V(x)=(\pi_2\otimes \pi_2)\delta_E(x),\\
&&\Delta_E: V\rightarrow \mathfrak{g}\otimes V,~~~~\Delta_E(x)=(\pi_1\otimes \pi_2)\delta_E(x),
\end{eqnarray*}
for any $x\in V$. Then it is easy to show that $\Omega^c(\mathfrak{g},\delta_\mathfrak{g}, V)=(\Delta_E, \Delta_V, \delta_V)$ is a Lie coalgebraic extending structure and $\varphi: E\rightarrow \mathfrak{g}\natural^c V$ given by $\varphi(a)=\pi_1(a)+\pi_2(a)$ for all $a\in E$ is a Lie coalgebra isomorphism.
\end{proof}
\begin{lemm}\label{l1}
Let $\Omega^c(\mathfrak{g},\delta,V)=(\Delta_E, \Delta_V, \delta_V)$ and ${\Omega^c}^{'}(\mathfrak{g},\delta,V)=(\Delta^{'}_E, \Delta^{'}_V, \delta^{'}_V)$ be two Lie coalgebraic extending structures of $(\mathfrak{g},\delta_\mathfrak{g})$ by $V$ and
$\mathfrak{g}\natural^cV$, $\mathfrak{g}{\natural^c}^{'}V$ be the corresponding unified co-products. Then there exists a bijection between the set of all homomorphisms of Lie coalgebras $\varphi: \mathfrak{g}\natural^cV\rightarrow \mathfrak{g}{\natural^c}^{'}V$ whose restriction on $\mathfrak{g}$ is the identity map and the set of pairs $(p,q)$, where $p:V\rightarrow \mathfrak{g}$ and
$q:V\rightarrow V$ are two linear maps satisfying
\begin{eqnarray}
\label{q1}&&\delta_g(p(x))+\Delta^{'}_V(q(x))=(I\otimes p)\Delta_E(x)-(p\otimes I)\tau\Delta_E(x)+(p\otimes p)\delta_V(x)+\Delta_V(x),\\
&&\Delta_E^{'}(q(x))=(I\otimes q)\Delta_E(x)+(p\otimes q)\delta_V(x),\\
\label{q2}&&\delta_V^{'}(q(x))=(q\otimes q)\delta_V(x),
\end{eqnarray}
for all $x\in V$.

The bijection from the homomorphism of Lie coalgebras $\varphi=\varphi_{p,q}: \mathfrak{g}\natural^cV\rightarrow \mathfrak{g}{\natural^c}^{'}V$ to $(p,q)$ is given by $\varphi(a+x)=a+p(x)+q(x)$ for all $a\in \mathfrak{g}$ and $x\in V$. Moreover, $\varphi=\varphi_{p,q}$ is an isomorphism if and only if $q: V\rightarrow V$ is a linear isomorphism.
\end{lemm}
\begin{proof}
Let $\varphi$ be a Lie coalgebra homomorphism from $\mathfrak{g}\natural^cV$ to $\mathfrak{g}{\natural^c}^{'}V$ whose restriction on $\mathfrak{g}$ is the identity map. Obviously, $\varphi$ is determined by two linear maps $p: V\rightarrow \mathfrak{g}$ and $q: V\rightarrow V$ such that
$\varphi(a+x)=a+p(x)+q(x)$ for all $a\in \mathfrak{g}$ and $x\in V$. Therefore, we only need to show that
$\varphi$ is a homomorphism of Lie coalgebras if and only if (\ref{q1})-(\ref{q2}) hold.  Note that $\delta^{'}_E\varphi(a)=(\varphi\otimes \varphi)\delta_E(a)$ for all $a\in \mathfrak{g}$. Then we consider that $\delta^{'}_E\varphi(x)=(\varphi\otimes \varphi)\delta_E(x)$ for all $x\in V$.
\begin{eqnarray*}
\delta^{'}_E\varphi(x)&=&\delta^{'}_E(p(x)+q(x))\\
&=&\delta_g(p(x))+\Delta_E^{'}(q(x))-\tau\Delta_E^{'}(q(x))+\Delta_V^{'}(q(x))+\delta_V^{'}(q(x)),
\end{eqnarray*}
and
\begin{eqnarray*}
(\varphi\otimes \varphi)\delta_E(x)&=&(\varphi\otimes \varphi)(\Delta_E(x)-\tau\Delta_E(x)+\Delta_V(x)+\delta_V(x))\\
&=&(I\otimes p)\Delta_E(x)+(I\otimes q)\Delta_E(x)-(p\otimes I)\tau\Delta_E(x)
-(q\otimes I)\tau\Delta_E(x)\\
&&+(I\otimes I)\Delta_V(x)+((p+q)\otimes (p+q))\delta_V(x).
\end{eqnarray*}
Then one can directly obtain that $\delta^{'}_E\varphi(x)=(\varphi\otimes \varphi)\delta_E(x)$ for all $x\in V$ if and only if (\ref{q1})-(\ref{q2}) hold. By the definition of $\varphi=\varphi_{p,q}$, it is easy to see that $\varphi=\varphi_{p,q}$ is an isomorphism if and only if $q: V\rightarrow V$ is a linear isomorphism.
\end{proof}
\begin{rema}
By (\ref{q2}), $q:(V, \delta_V)\rightarrow (V, \delta_V^{'})$ is a homomorphism of Lie coalgebras.
\end{rema}
\begin{defi}
Let $(\mathfrak{g},\delta_\mathfrak{g})$ be a Lie coalgebra and $V$ a vector space. Two Lie coalgebraic extending structures $\Omega^c(\mathfrak{g},\delta_\mathfrak{g},V)=(\Delta_E, \Delta_V, \delta_V)$ and ${\Omega^c}^{'}(\mathfrak{g},\delta_\mathfrak{g},V)=(\Delta^{'}_E, \Delta^{'}_V, \delta^{'}_V)$ of $(\mathfrak{g},\delta_\mathfrak{g})$ by $V$ are called {\bf equivalent}, if there exists a pair $(p,q)$ of linear maps, where $p: V\rightarrow \mathfrak{g}$ and $q\in Aut_k(V)$ such that $(\Delta^{'}_E, \Delta^{'}_V, \delta^{'}_V)$ can be obtained from $(\Delta_E, \Delta_V, \delta_V)$ by $(p,q)$ as follows:
\begin{eqnarray}
&&\label{equ1}\delta_V^{'}(x)=(q\otimes q)\delta_V(q^{-1}(x)),\\
&&\label{equ2}\Delta_E^{'}(x)=(I\otimes q)\Delta_E(q^{-1}(x))+(p\otimes q)\delta_V(q^{-1}(x)),\\
&&\Delta^{'}_V(x)=(I\otimes p)\Delta_E(q^{-1}(x))-(p\otimes I)\tau\Delta_E(q^{-1}(x))\nonumber\\
&&\label{equ3}+(p\otimes p)\delta_V(q^{-1}(x))+\Delta_V(q^{-1}(x))-\delta_\mathfrak{g}(p(q^{-1}(x))),
\end{eqnarray}
for all $x\in V$. We denote it by $\Omega^c(\mathfrak{g},\delta_\mathfrak{g},V)\equiv{\Omega^c}^{'}(\mathfrak{g},\delta_\mathfrak{g},V)$.
\end{defi}
\begin{theo}
Let $(\mathfrak{g},\delta_\mathfrak{g})$ be a Lie coalgebra, $E$ a vector space containing $\mathfrak{g}$ as a subspace and
$V$ be a $\mathfrak{g}$-complement in $E$. Denote $\mathcal{HC}_{\mathfrak{g}}^2(V,\mathfrak{g}):=\mathcal{LC}(\mathfrak{g},V)/\equiv$. Then the map
\begin{eqnarray}
\mathcal{HC}_{\mathfrak{g}}^2(V,\mathfrak{g})\rightarrow CExtd(E,\mathfrak{g}),~~~~\overline{\Omega(\mathfrak{g},\delta, V)}\rightarrow (\mathfrak{g}\natural^c V,\delta_E)
\end{eqnarray}
is bijective, where $\overline{(\mathfrak{g},\delta_\mathfrak{g}, V)}$ is the equivalence class of $\Omega(\mathfrak{g},\delta_\mathfrak{g}, V)$ under $\equiv$.
\end{theo}
\begin{proof}
It can be directly obtained by Theorem \ref{t2} and Lemma \ref{l1}.
\end{proof}

\section{Unified bi-products for Lie bialgebras}
In this section, we present a general theory of extending structures of Lie bialgebras using those results in the context of Lie algebras and Lie coalgebras.

\begin{defi}
Let $(\mathfrak{g}, [\cdot,\cdot], \delta_\mathfrak{g})$ be a Lie bialgebra and $V$ a vector space. An {\bf extending datum} of $(\mathfrak{g}, [\cdot,\cdot], \delta_\mathfrak{g})$ by $V$ is a system $\Omega^{bi}(\mathfrak{g},V)=(\lhd, \rhd, f, \{\cdot,\cdot\}, \Delta_E, \Delta_V, \delta_V)$ consisting four bilinear maps and three linear maps
\begin{eqnarray*}
&&\lhd: V\times \mathfrak{g}\rightarrow V,~~~~\rhd: V\times \mathfrak{g}\rightarrow \mathfrak{g},~~~~f:  V\times V \rightarrow  \mathfrak{g},~~~ \{\cdot,\cdot\}:V\times V \rightarrow V,\\
&&\Delta_E: V\rightarrow  \mathfrak{g}\otimes V,~~~~\Delta_V: V\rightarrow  \mathfrak{g}\otimes \mathfrak{g},~~~~\delta_V: V\rightarrow V\otimes V.
\end{eqnarray*}
Suppose that $\Omega^{bi}(\mathfrak{g},V)$ is an extending datum of $(\mathfrak{g}, [\cdot,\cdot], \delta_\mathfrak{g})$ by $V$. Denote by $\mathfrak{g}\natural_{\Omega^{bi}(\mathfrak{g},V)}V=\mathfrak{g}\natural^{bi}V$ the vector space $E=\mathfrak{g}\oplus V$ with the bilinear map
$[\cdot,\cdot]: E\times E \rightarrow E$ given by (\ref{qL1}) and the linear map
$\delta_E: E\rightarrow E\otimes E$ given by (\ref{eq1}).
If $\mathfrak{g}\natural^{bi}V$ is a Lie bialgebra with the Lie bracket given by (\ref{qL1}) and the Lie co-bracket given by (\ref{eq1}), then $\mathfrak{g}\natural^{bi}V$ is called
the {\bf unified bi-product} of $(\mathfrak{g}, [\cdot,\cdot], \delta_\mathfrak{g})$ and $\Omega^{bi}(\mathfrak{g},V)$, and $\Omega^{bi}(\mathfrak{g},V)$ is called a {\bf Lie bialgebraic extending structure} of $(\mathfrak{g}, [\cdot,\cdot], \delta_\mathfrak{g})$ by $V$. We denote the set of Lie bialgebraic extending structure of $(\mathfrak{g}, [\cdot,\cdot], \delta_\mathfrak{g})$ by $V$  by $\mathcal{LBI}(\mathfrak{g},V)$.
\end{defi}
\begin{theo}
Let $(\mathfrak{g},[\cdot,\cdot],\delta_{\mathfrak{g}})$ be a Lie bialgebra, $V$  a vector space and $\Omega^{bi}(\mathfrak{g},V)$ be an extending datum of
$(\mathfrak{g},[\cdot,\cdot], \delta_{\mathfrak{g}})$ by $V$. Then $\mathfrak{g}\natural^{bi} V$ is a unified bi-product if and only if the following
compatibility conditions hold:
\begin{eqnarray*}
(BE1)~~&&\text{$(\lhd,\rhd,f,\{\cdot,\cdot\})$ is a Lie extending structure of $\mathfrak{g}$ by $V$, and $(\Delta_E, \Delta_V, \delta_V)$ }is\\
 &&\text{a Lie coalgebraic extending structure of $(\mathfrak{g},\delta_\mathfrak{g})$ by $V$;}\\
(BE2)~~ &&-\Delta_V(x\lhd a)-\delta_\mathfrak{g}(x\rhd a)=(\tau-I\otimes I)(I\otimes R_\rhd(a))\Delta_E(x)+a.\Delta_V(x)\\
&&-(L_{\rhd}(x)\otimes I+I\otimes L_{\rhd}(x))\delta_\mathfrak{g}(a),\\
(BE3)~~&&\Delta_E(x\lhd a)=(-ad(a)\otimes I+I\otimes R_\lhd(a))\Delta_E(x)+(R_\rhd(a)\otimes I)\delta_V(x)+(I\otimes L_\lhd(x))\delta_\mathfrak{g}(a),\\
(BE4)~~&&\delta_V(x\lhd a)=(I\otimes R_\lhd(a)+R_\lhd(a)\otimes I)\delta_V(x),\\
(BE5)~~&&\delta_\mathfrak{g}(f(x,y))+\Delta_V(\{x,y\})
=(I\otimes I-\tau)(I\otimes f(x,\cdot))\Delta_E(y)\\
&&+(L_\rhd(x)\otimes I+I\otimes L_\rhd(x) )\Delta_V(y)
-(I\otimes I-\tau)(I\otimes f(y,\cdot))\Delta_E(x)\\
&&-(L_\rhd(y)\otimes I+I\otimes L_\rhd(y))\Delta_V(x),\\
(BE6)~~&&\Delta_E(\{x,y\})=(L_\rhd(x)\otimes I+I\otimes \{x,\cdot\})\Delta_E(y)
+(I\otimes L_\lhd(x))\Delta_V(y)\\
&&+(f(x,\cdot)\otimes I)\delta_V(y)
-(L_\rhd(y)\otimes I+I\otimes \{y,\cdot\})\Delta_E(x)\\
&&-(I\otimes L_\lhd(y))\Delta_V(x)-(f(y,\cdot)\otimes I)\delta_V(x),\\
(BE7)~~&&\delta_V(\{x,y\})=(I\otimes I-\tau)(L_\lhd(x)\otimes I)\Delta_E(y)
+(\{x,\cdot\}\otimes I+I\otimes \{x,\cdot\})\delta_V(y)\\
&&-(I\otimes I-\tau)(L_\lhd(y)\otimes I)\Delta_E(x)
-(\{y,\cdot\}\otimes I+I\otimes \{y,\cdot\})\delta_V(x),\end{eqnarray*}
for all $a\in \mathfrak{g}$ and $x$, $y\in V$, where $L_\rhd(x)a=R_\rhd(a)x=x\rhd a$,
$L_\lhd(x)a=R_\lhd(a)x=x\lhd a$ and $ad(a)b=[a,b]$ for all $a\in \mathfrak{g}$.
\end{theo}
\begin{proof}
By the definitions of unified product for Lie algebras and unified coproduct for Lie coalgebras,
$(E, [\cdot,\cdot])$ is a Lie algebra if and only if $(\lhd,\rhd,f,\{\cdot,\cdot\})$ is a Lie extending structure of $\mathfrak{g}$ by $V$, and $(E, \delta_E)$ is a Lie coalgebra if and only if $(\Delta_E, \Delta_V, \delta_V)$ is a Lie coalgebraic extending structure of  $(\mathfrak{g},\delta_\mathfrak{g})$ by $V$.
Then we only need to consider that $\delta_E([a+x,b+y])=(a+x).\delta_E(b+y)-(b+y).\delta_E(a+x)$ for all
$a$, $b\in \mathfrak{g}$, and $x$, $y\in V$. Obviously, $\delta_E([a,b])=a.\delta_E(b)-b.\delta_E(a)$ holds for all $a$, $b\in \mathfrak{g}$. By skew-symmetry of Lie brackets, it is enough to prove that
$\delta_E([a,x])=a.\delta_E(x)-x.\delta_E(a)$ and $\delta_E([x,y])=x.\delta_E(y)-y.\delta_E(x)$ for all
$a\in \mathfrak{g}$, and $x$, $y\in V$ if and only if $(BE2)$-$(BE7)$ hold.

Since
\begin{eqnarray*}
\delta_E([a,x])&=&-\delta_E(x\rhd a+x\lhd a)\\
&=&-\delta_\mathfrak{g}(x\rhd a)-\Delta_E(x\lhd a)+\tau\Delta_E(x\lhd a)-\Delta_V(x\lhd a)-\delta_V(x\lhd a),
\end{eqnarray*}
and
\begin{eqnarray*}
&&a.\delta_E(x)-x.\delta_E(a)\\
&&=(ad(a)\otimes I-I\otimes R_{\rhd}(a)-I\otimes R_{\lhd}(a))\Delta_E(x)
-(-R_{\rhd}(a)\otimes I-R_{\lhd}(a)\otimes I+I\otimes ad(a))\tau\Delta_E(x)\\
&&+a.\Delta_V(x)
-(I\otimes R_\rhd(a)+I\otimes R_{\lhd}(a)+R_{\rhd}(a)\otimes I+R_\lhd(a)\otimes I)\delta_V(x)\\
&&-(L_\rhd(x)\otimes I+L_\lhd(x)\otimes I+I\otimes L_\rhd(x)+I\otimes L_\lhd(x))\delta_\mathfrak{g}(a),
\end{eqnarray*}
it is easy to get that $\delta_E([a,x])=a.\delta_E(x)-x.\delta_E(a)$ for all
$a\in \mathfrak{g}$, and $x\in V$ if and only if $(BE2)$-$(BE4)$ hold.

Similarly, by
\begin{eqnarray*}
\delta_E([x,y])&=&\delta_E(f(x,y)+\{x,y\})\\
&=&\delta_\mathfrak{g}(f(x,y))+\Delta_E(\{x,y\})-\tau\Delta_E(\{x,y\})
+\Delta_V(\{x,y\})+\delta_V(\{x,y\}),
\end{eqnarray*}
and
\begin{eqnarray*}
&&x.\delta_E(y)-y.\delta_E\delta(x)\\
&&=x.(\Delta_E(y)-\tau\Delta_E(y)+\Delta_V(y)+\delta_V(y))\\
&&-y.(\Delta_E(x)-\tau\Delta_E(x)+\Delta_V(x)+\delta_V(x))\\
&&=((L_\lhd(x)+L_\rhd(x))\otimes I+I\otimes f(x,\cdot)+I\otimes \{x,\cdot\})\Delta_E(y)\\
&&-(I\otimes (L_\lhd(x)+L_\rhd(x))+f(x,\cdot)\otimes I+ \{x,\cdot\}\otimes I)\tau\Delta_E(y)\\
&&+((L_\lhd(x)+L_\rhd(x))\otimes I+I\otimes (L_\lhd(x)+L_\rhd(x)))\Delta_V(y)\\
&&+(f(x,\cdot)\otimes I+\{x,\cdot\}\otimes I+I\otimes f(x,\cdot)+I\otimes \{x,\cdot\})\delta_V(y)\\
&&-((L_\lhd(y)+L_\rhd(y))\otimes I+I\otimes f(y,\cdot)+I\otimes \{y,\cdot\})\Delta_E(x)\\
&&+(I\otimes (L_\lhd(y)+L_\rhd(y))+f(y,\cdot)\otimes I+ \{y,\cdot\}\otimes I)\tau\Delta_E(x)\\
&&-((L_\lhd(y)+L_\rhd(y))\otimes I+I\otimes (L_\lhd(y)+L_\rhd(y)))\Delta_V(x)\\
&&-(f(y,\cdot)\otimes I+\{y,\cdot\}\otimes I+I\otimes f(y,\cdot)+I\otimes \{y,\cdot\})\delta_V(x),
\end{eqnarray*}
it is easy to see that $\delta_E([x,y])=x.\delta_E(y)-y.\delta_E(x)$ for all
$x$, $y\in V$ if and only if $(BE5)$-$(BE7)$ hold.
\end{proof}

Note that $\mathfrak{g}\natural^{bi} V$ is a Lie bialgebra containing $(\mathfrak{g},[\cdot,\cdot],\delta_\mathfrak{g})$ as a Lie sub-bialgebra. Next, we show that any Lie bialgebraic structure on $E$ containing $(\mathfrak{g},[\cdot,\cdot],\delta_\mathfrak{g})$ as a Lie sub-bialgebra is isomorphic to such a unified bi-product.
\begin{theo}\label{t3}
Let $(\mathfrak{g},[\cdot,\cdot],\delta_\mathfrak{g})$ be a Lie bialgebra and $E$ a vector space containing
$\mathfrak{g}$ as a subspace. Suppose that there is a Lie bialgebraic structure $(E, [\cdot,\cdot],\delta_E)$ on $E$ such that $(\mathfrak{g},[\cdot,\cdot],\delta_\mathfrak{g})$
is a Lie sub-bialgebra of $E$. Then there exists a Lie bialgebraic extending structure $\Omega^{bi}(\mathfrak{g},V)=(\lhd, \rhd, f, \{\cdot,\cdot\}, \Delta_E, \Delta_V, \delta_V)$ of $(\mathfrak{g},[\cdot,\cdot],\delta_\mathfrak{g})$ by $V$ such that
$(E, [\cdot,\cdot],\delta_E)\cong \mathfrak{g}\natural^{bi}V$.
\end{theo}
\begin{proof}
It can be directly obtained by Theorem \ref{t2} and Theorem 2.4 in \cite{AM1}.
\end{proof}

\begin{lemm}\label{l4}
Let $\Omega^{bi}(\mathfrak{g},V)=(\lhd, \rhd, f, \{\cdot,\cdot\}, \Delta_E, \Delta_V, \delta_V)$ and ${\Omega^{bi}}^{'}(\mathfrak{g},V)=(\lhd^{'}, \rhd^{'}, f^{'}, \{\cdot,\cdot\}^{'}, \Delta^{'}_E, \Delta^{'}_V, \delta^{'}_V)$ be two Lie bialgebraic extending structures of $(\mathfrak{g},[\cdot,\cdot], \delta_\mathfrak{g})$ by $V$ and
$\mathfrak{g}\natural^{bi}V$, $\mathfrak{g}{\natural^{bi}}^{'}V$ be the corresponding unified bi-products. Then there exists a bijection between the set of all homomorphisms of Lie bialgebras $\varphi: \mathfrak{g}\natural^{bi}V\rightarrow \mathfrak{g}{\natural}^{bi}V$ whose restriction on $\mathfrak{g}$ is the identity map and the set of pairs $(p,q)$, where $p:V\rightarrow \mathfrak{g}$ and
$q:V\rightarrow V$ are two linear maps satisfying (\ref{q1})-(\ref{q2}), and
\begin{eqnarray*}
&&q(x)\lhd^{'} a=q(x\lhd a),\\
&&p(x\lhd a)=[p(x),a]-x\rhd a+ q(x)\rhd^{'} a,\\
&&q(\{x,y\})=\{q(x),q(y)\}^{'}+q(x)\lhd^{'}p(y)-q(y)\lhd^{'}p(x),\\
&&p(\{x,y\})=[p(x),p(y)]+q(x)\rhd^{'}p(y)-q(y)\rhd^{'}p(x)+f^{'}(q(x),q(y))-f(x,y)
\end{eqnarray*}
for all $a\in\mathfrak{g}$ and $x$, $y\in V$.

The bijection from the homomorphism of Lie bialgebras $\varphi=\varphi_{p,q}: \mathfrak{g}\natural^{bi}V\rightarrow \mathfrak{g}{\natural^{bi}}^{'}V$ to $(p,q)$ is given  by $\varphi(a+x)=a+p(x)+q(x)$ for all $a\in \mathfrak{g}$ and $x\in V$. Moreover, $\varphi=\varphi_{p,q}$ is an isomorphism if and only if $q: V\rightarrow V$ is a linear isomorphism.
\end{lemm}
\begin{proof}
Since a homomorphism of Lie bialgebras is a homomorphism of Lie algebras and also a homomorphism of Lie coalgebras, it can be  directly obtained by Lemma \ref{l1} and Lemma 2.5 in \cite{AM1}.
\end{proof}

\begin{defi}
Let $(\mathfrak{g},[\cdot,\cdot], \delta_\mathfrak{g})$ be a Lie bialgebra and $V$ a vector space. Two Lie bialgebra extending structures $\Omega^{bi}(\mathfrak{g},V)=(\lhd, \rhd, f, \{\cdot,\cdot\}, \Delta_E, \Delta_V, \delta_V)$ and ${\Omega^{bi}}^{'}(\mathfrak{g},V)=(\lhd^{'}, \rhd^{'}, f^{'}, \{\cdot,\cdot\}^{'}, \Delta^{'}_E, \Delta^{'}_V, \delta^{'}_V)$  of $(\mathfrak{g},[\cdot,\cdot],\delta_\mathfrak{g})$ by $V$ are called {\bf equivalent}, if there exists a pair $(p,q)$ of linear maps, where $p: V\rightarrow \mathfrak{g}$ and $q\in Aut_k(V)$ such that $(\Delta^{'}_E, \Delta^{'}_V, \delta^{'}_V)$ can be obtained from $(\Delta_E, \Delta_V, \delta_V)$ by $(p,q)$ as (\ref{equ1})-(\ref{equ3}) and
\begin{eqnarray}
&&x\lhd^{'} a=q(q^{-1}(x)\lhd a),\\
&&x\rhd^{'} a=p(q^{-1}(x)\lhd a)+q^{-1}(x) \rhd a-[p(q^{-1}(x)),a],\\
&&f^{'}(x,y)=f(q^{-1}(x),q^{-1}(y))+p(\{q^{-1}(x),q^{-1}(y)\})-[p(q^{-1}(x)),p(q^{-1}(y))]\nonumber\\
&&-p(q^{-1}(x)\lhd p(q^{-1}(y)))-q^{-1}(x)\rhd p(q^{-1}(y))+p(q^{-1}(y)\lhd p(q^{-1}(x)))+q^{-1}(y)\rhd p(q^{-1}(x)),\\
&&\{x,y\}^{'}=q(\{q^{-1}(x),q^{-1}(y)\})-q(q^{-1}(x)\lhd p(q^{-1}(y)))+q(q^{-1}(y)\lhd p(q^{-1}(x)))
\end{eqnarray}
for all $a\in \mathfrak{g}$ and $x$, $y\in V$. We denote $\Omega^{bi}(\mathfrak{g},V)\equiv {\Omega^{bi}}^{'}(\mathfrak{g},V)$.
\end{defi}

\begin{theo}\label{tt1}
Let $(\mathfrak{g}, [\cdot,\cdot], \delta_\mathfrak{g})$ be a Lie bialgebra, $E$ a vector space containing $\mathfrak{g}$ as a subspace and
$V$ be a complement of $\mathfrak{g}$ in $E$.
Denote $\mathcal{HBI}_{\mathfrak{g}}^2(V,\mathfrak{g}):=\mathcal{LBI}(\mathfrak{g},V)/\equiv$. Then the map
\begin{eqnarray}
\mathcal{HBI}_{\mathfrak{g}}^2(V,\mathfrak{g})\rightarrow BExtd(E,\mathfrak{g}),~~~~\overline{\Omega^{bi}(\mathfrak{g},V)}\rightarrow (\mathfrak{g}\natural^{bi} V,\delta_E)
\end{eqnarray}
is bijective, where $\overline{\Omega^{bi}(\mathfrak{g}, V)}$ is the equivalence class of $\Omega^{bi}(\mathfrak{g}, V)$ under $\equiv$.
\end{theo}
\begin{proof}
It can be directly obtained by Theorem \ref{t3} and Lemma \ref{l4}.
\end{proof}
\section{Special unified bi-products}
In this section, we will present several special cases of unified bi-products for Lie bialgebras.
\subsection{Crossed bi-products of Lie bialgebras}
Let $\Omega^{bi}(\mathfrak{g}, V)=(\lhd, \rhd, f, \{\cdot, \cdot\}, \Delta_E, \Delta_V, \delta_V)$ be an extending datum of $(\mathfrak{g}, [\cdot,\cdot], \delta_\mathfrak{g})$ by $V$ with $\lhd$ trivial. Denote this datum by $\Omega^{bi}(\mathfrak{g}, V)=(\rhd, f, \{\cdot, \cdot\}, \Delta_E, \Delta_V, \delta_V)$. Then in this case, $\Omega^{bi}(\mathfrak{g}, V)$ is a Lie bialgebraic extending structure of $(\mathfrak{g}, [\cdot,\cdot], \delta_\mathfrak{g})$ by $V$ if and only if $(V, \{\cdot,\cdot\}, \delta_V)$ is a Lie bialgebra, $f(x,x)=0$, $(LE6)$, $(CLE1)-(CLE5)$, $(BE5)$ and
\begin{eqnarray}
&&x \rhd [a,b]=[x\rhd a, b]+[a, x\rhd b],\\
&&\{x,y\}\rhd a=x\rhd(y\rhd a)-y\rhd (x\rhd a)+[a,f(x,y)],\\
 &&-\delta_\mathfrak{g}(x\rhd a)=(\tau-I\otimes I)(I\otimes R_\rhd(a))\Delta_E(x)+a.\Delta_V(x)\nonumber\\
 &&-(L_{\rhd}(x)\otimes I+I\otimes L_{\rhd}(x))\delta_\mathfrak{g}(a),\\
&&-(ad(a)\otimes I)\Delta_E(x)+(R_\rhd(a)\otimes I)\delta_V(x)=0,\\
&&\Delta_E(\{x,y\})=(L_\rhd(x)\otimes I+I\otimes \{x,\cdot\})\Delta_E(y)+(f(x,\cdot)\otimes I)\delta_V(y)\nonumber\\
&&-(L_\rhd(y)\otimes I+I\otimes \{y,\cdot\})\Delta_E(x)-(f(y,\cdot)\otimes I)\delta_V(x),
\end{eqnarray}
for $a$, $b\in \mathfrak{g}$ and $x$, $y\in V$. Denote the corresponding unified bi-product $\mathfrak{g}{\natural}^{bi}V$ by
$\mathfrak{g}\lozenge_{\rhd, f}^{\Delta_E, \Delta_V}V$, which is called the {\bf crossed bi-product} of the Lie bialgebras $(\mathfrak{g}, [\cdot,\cdot], \delta_\mathfrak{g})$ and $(V, \{\cdot,\cdot\}, \delta_V)$. Note that $\mathfrak{g}\lozenge_{\rhd, f}^{\Delta_E, \Delta_V}V$ is a Lie bialgebra with the following Lie bracket and Lie co-bracket
\begin{eqnarray}
&&[a+x, b+y]=[a,b]+x\rhd b-y\rhd a+f(x,y)+\{x,y\},\\
&&\delta_E(a+x)=\delta_\mathfrak{g}(a)+\Delta_E(x)-\tau\Delta_E(x)+\Delta_V(x)+\delta_V(x),
\end{eqnarray}
for all $a$, $b\in \mathfrak{g}$, $x$, $y\in V$. In this case, $(\mathfrak{g}, [\cdot,\cdot], \delta_\mathfrak{g})$ is a Lie sub-bialgebra of $\mathfrak{g}\lozenge_{\rhd, f}^{\Delta_E, \Delta_V}V$ and $\mathfrak{g}$ is an ideal of the Lie algebra $\mathfrak{g}\lozenge_{\rhd, f}^{\Delta_E, \Delta_V}V$.

\begin{prop}
Let $(\mathfrak{g}, [\cdot,\cdot], \delta_\mathfrak{g})$ be a Lie bialgebra, $E$ a vector space containing $\mathfrak{g}$ as a subspace.
Then any Lie bialgebraic structure on $E$ that contains $(\mathfrak{g}, [\cdot,\cdot], \delta_\mathfrak{g})$ as a Lie sub-bialgebra and $(\mathfrak{g}, [\cdot,\cdot])$ as an ideal of the Lie algebra $E$ is isomorphic to a crossed biproduct of Lie bialgebras $\mathfrak{g}\lozenge_{\rhd, f}^{\Delta_E, \Delta_V}V$.
\end{prop}
\begin{proof}
It can be directly obtained by Theorem \ref{t3}.
\end{proof}

In particular, when $f$ and $\Delta_V$ are trivial, we denote $\mathfrak{g}\lozenge_{\rhd, f}^{\Delta_E, \Delta_V}V$ by $\mathfrak{g}\lozenge_{\rhd}^{\Delta_E}V$ simply. Then $\mathfrak{g}\lozenge_{\rhd}^{\Delta_E}V$ is a Lie bialgebra with the following Lie bracket and Lie co-bracket as follows
\begin{eqnarray*}
&&[a+x, b+y]=[a,b]+x\rhd b-y\rhd a+\{x,y\},\\
&&\delta_E(a+x)=\delta_\mathfrak{g}(a)+\Delta_E(x)-\tau\Delta_E(x)+\delta_V(x),~~~~a,~b\in \mathfrak{g}, ~~~x,~~y\in V,
\end{eqnarray*}
if and only if $(V, \{\cdot,\cdot\}, \delta_V)$ is a Lie bialgebra, $(\mathfrak{g},\rhd)$ is a left $V$-module, and
\begin{eqnarray}
&&x \rhd [a,b]=[x\rhd a, b]+[a, x\rhd b],\\
&&\{x,y\}\rhd a=x\rhd(y\rhd a)-y\rhd (x\rhd a),\\
&&(I\otimes \Delta_E)\Delta_E(x)-\tau_{12}(I\otimes \Delta_E)\Delta_E(x)=(\delta_\mathfrak{g}\otimes I)\Delta_E(x),\\
&&(I\otimes \delta_V)\Delta_E(x)-\tau_{12}(I\otimes \Delta_E)\delta_V(x)=(\Delta_E\otimes I)\delta_V(x),\\
 &&\delta_\mathfrak{g}(x\rhd a)=-(\tau-I\otimes I)(I\otimes R_\rhd(a))\Delta_E(x)+(L_{\rhd}(x)\otimes I+I\otimes L_{\rhd}(x))\delta_\mathfrak{g}(a),\\
&&-(ad(a)\otimes I)\Delta_E(x)+(R_\rhd(a)\otimes I)\delta_V(x)=0,\\
&&\Delta_E(\{x,y\})=(L_\rhd(x)\otimes I+I\otimes \{x,\cdot\})\Delta_E(y)-(L_\rhd(y)\otimes I+I\otimes \{y,\cdot\})\Delta_E(x),
\end{eqnarray}
for all $x$, $y\in V$ and $a\in \mathfrak{g}$. $\mathfrak{g}\lozenge_{\rhd}^{\Delta_E}V$ is call {\bf the left-right bicrossed sum Lie bialgebra} (see Proposition 8.3.5 in \cite{Mj}).

\subsection{Double cross sum of Lie bialgebras}
Let $\Omega^{bi}(\mathfrak{g}, V)=(\lhd, \rhd, f, \{\cdot, \cdot\}, \Delta_E, \Delta_V, \delta_V)$ be an extending datum of $(\mathfrak{g}, [\cdot,\cdot], \delta_\mathfrak{g})$ by $V$ with $f$, $\Delta_E$ and $\Delta_V$ trivial. Denote this datum by $\Omega^{bi}(\mathfrak{g}, V)=(\lhd, \rhd, \{\cdot, \cdot\}, \delta_V)$. Then in this case, $\Omega^{bi}(\mathfrak{g}, V)$ is a Lie bialgebraic extending structure of $(\mathfrak{g}, [\cdot,\cdot], \delta_\mathfrak{g})$ by $V$ if and only if $(V, \{\cdot,\cdot\}, \delta_V)$ is a Lie bialgebra, $(\mathfrak{g},\rhd)$ is a left $V$-module, $(V,\lhd)$ is a right $\mathfrak{g}$-module, $(LE3)$, $(LE4)$, $(BE4)$ and
\begin{eqnarray}
&&\{x,y\}\rhd a=x\rhd(y\rhd a)-y\rhd (x\rhd a),\\
 &&\delta_\mathfrak{g}(x\rhd a)=(L_{\rhd}(x)\otimes I+I\otimes L_{\rhd}(x))\delta_\mathfrak{g}(a),\\
&& (R_\rhd(a)\otimes I)\delta_V(x)+(I\otimes L_{\lhd}(x))\delta_\mathfrak{g}(x)=0,
\end{eqnarray}
for all $a\in \mathfrak{g}$ and $x$, $y\in V$. Denote the corresponding unified bi-product $\mathfrak{g}{\natural}^{bi}V$ by
$\mathfrak{g}\lozenge_{\lhd, \rhd}^{bi}V$, which is called the {\bf double cross sum} of the Lie bialgebras $(\mathfrak{g}, [\cdot,\cdot], \delta_\mathfrak{g})$ and $(V, \{\cdot,\cdot\}, \delta_V)$ (see Proposition 8.3.4 in \cite{Mj}). Note that $\mathfrak{g}\lozenge_{\lhd, \rhd}^{bi}V$ is a Lie bialgebra with the following Lie bracket and Lie co-bracket
\begin{eqnarray}
&&[a+x, b+y]=[a,b]+x\rhd b-y\rhd a+x\lhd b-y\lhd a+\{x,y\},\\
&&\delta_E(a+x)=\delta_\mathfrak{g}(a)+\delta_V(x),
\end{eqnarray}
for all $a$, $b\in \mathfrak{g}$, $x$, $y\in V$. In this case, $(\mathfrak{g}, [\cdot,\cdot], \delta_\mathfrak{g})$ and $(V, \{\cdot,\cdot\}, \delta_V)$ are two Lie sub-bialgebras of $\mathfrak{g}\lozenge_{\lhd, \rhd}^{bi}V$.

\begin{prop}
Let $(\mathfrak{g}, [\cdot,\cdot], \delta_\mathfrak{g})$ and $(V, \{\cdot,\cdot\}, \delta_V)$ be two Lie bialgebras, $E$ a vector space containing $\mathfrak{g}$ and $V$ as subspaces.
Then any Lie bialgebra structure on $E=\mathfrak{g}\oplus V$ as a vector space that contains $(\mathfrak{g}, [\cdot,\cdot], \delta_\mathfrak{g})$ and $(V, \{\cdot,\cdot\}, \delta_V)$ as two Lie sub-bialgebras  is isomorphic to a double cross sum of Lie bialgebras $\mathfrak{g}\lozenge_{\lhd, \rhd}^{bi}V$.
\end{prop}
\begin{proof}
It can be directly obtained by Theorem \ref{t3} and Lemma \ref{l4}.
\end{proof}
\section{Unified bi-products when $\text{dim}V=1$}
In this section, we investigate the unified bi-product associated with $(\mathfrak{g},[\cdot, \cdot], \delta_{\mathfrak{g}})$ and a one-dimensional vector space $V$ in detail.
\begin{defi}
Let $(\mathfrak{g},[\cdot, \cdot], \delta_{\mathfrak{g}})$ be a Lie bialgebra. A {\bf flag datum} of $(\mathfrak{g},[\cdot, \cdot], \delta_{\mathfrak{g}})$ is a
4-tuple $(\alpha, D, A, B)$, where $A\in \mathfrak{g}$,
$B=\sum B_{(1)}\otimes B_{(2)}\in \mathfrak{g}\wedge\mathfrak{g}$,
$\alpha: \mathfrak{g}\rightarrow k$ and $D:\mathfrak{g}\rightarrow \mathfrak{g}$ are two linear maps
satisfying the following identities
\begin{eqnarray}
\label{qe1}&&\alpha([a,b])=0,~~\delta_{\mathfrak{g}}(A)=0,~~[a,A]=\sum \alpha(a_{(2)})a_{(1)},\\
&&\label{qe12}D([a,b])=[D(a),b]+[a,D(b)]+\alpha(a)D(b)-\alpha(b)D(a),\\
&&\label{qe13}(A\otimes B-\tau_{12}(A\otimes B)+B\otimes A)+(I\otimes \delta_{\mathfrak{g}}-\tau_{12}(I\otimes \delta_{\mathfrak{g}})-(\delta_{\mathfrak{g}}\otimes I))B=0,\\
&&D(a)\otimes A-A\otimes D(a)+\alpha(a)B+a.B\nonumber\\
&&\label{qe2}+\delta_{\mathfrak{g}}(D(a))-\sum D(a_{(1)})\otimes a_{(2)}-\sum a_{(1)}\otimes D(a_{(2)})=0,
\end{eqnarray}
for all $a$, $b\in \mathfrak{g}$, where $\delta_{\mathfrak{g}}(a)=\sum a_{(1)}\otimes a_{(2)}$.
\end{defi}
Denote the set of all flag datums of $(\mathfrak{g},[\cdot, \cdot], \delta_{\mathfrak{g}})$ by $\mathcal{FLB}(\mathfrak{g},\delta_{\mathfrak{g}})$. In particular,
denote the set of all flag datums of the form $(0,D,0,B)$ by $\mathcal{TFLB}(\mathfrak{g},\delta_{\mathfrak{g}})$, where $D$ is a derivation of $\mathfrak{g}$.

\begin{prop}\label{ppr1}
Let $(\mathfrak{g},[\cdot, \cdot], \delta_{\mathfrak{g}})$ be a Lie bialgebra and $V=kx$ be a one-dimensional vector space. Then there is a bijection between the set $\mathcal{LBI}(\mathfrak{g},V)$ of all Lie bialgebraic structures of $(\mathfrak{g},[\cdot, \cdot], \delta_{\mathfrak{g}})$ by $V$  and $\mathcal{FLB}(\mathfrak{g},\delta_{\mathfrak{g}})$.
\end{prop}
\begin{proof}
Let  $\Omega^{bi}(\mathfrak{g},V)=(\lhd, \rhd, f, \{\cdot,\cdot\}, \Delta_E, \Delta_V, \delta_V)$ be a Lie bialgebraic extending structure. By $(LE1)$ and $(CLE1)$, $f$, $\{\cdot,\cdot\}$ and $\delta_V$ are trivial. According to that $V$ is a vector space with a basis $\{x\}$, we can set
\begin{eqnarray*}
x\lhd a= \alpha(a)x,~~~x\rhd a=D(a),~~~\Delta_E(x)=A\otimes x,~~~\Delta_V(x)=B,
\end{eqnarray*}
where $A\in \mathfrak{g}$, $B\in \mathfrak{g}\otimes \mathfrak{g}$,
$\alpha: \mathfrak{g}\rightarrow k$ and $D: \mathfrak{g}\rightarrow \mathfrak{g}$ are two linear maps.
Since $\Delta_V(x)=-\tau \Delta_V(x)$, $B\in \mathfrak{g}\wedge \mathfrak{g}$. Then it is easy to check that $(BE1)$-$(BE7)$ hold if and only if (\ref{qe1})-(\ref{qe2}) are satisfied.
\end{proof}
\begin{defi}\label{defi6}
Two flag datums $(\alpha, D, A, B)$ and $(\alpha^{'}, D^{'}, A^{'}, B^{'}) \in \mathcal{FLB}(\mathfrak{g},\delta_{\mathfrak{g}})$ are called {\bf equivalent} if $\alpha^{'}=\alpha$, $A^{'}=A$ and there exist
some element $U\in \mathfrak{g}$ and $\beta\in k^{*}$ such that
\begin{eqnarray}
\label{eqq1}&& D(a)=[U,a]+\beta D^{'}(a)-\alpha(a)U,\\
\label{eqq2}&&B=\delta_{\mathfrak{g}}(U)+\beta B^{'}+U\otimes A-A\otimes U.
\end{eqnarray}
Denote it by $(\alpha, D, A, B)\equiv (\alpha^{'}, D^{'}, A^{'}, B^{'})$. Similarly, if (\ref{eqq1}) (resp. (\ref{eqq2})) holds, then we denote it by $D\equiv  D^{'}$ (resp. $B\equiv  B^{'}$).
\end{defi}

\begin{theo}\label{th6}
Let $(\mathfrak{g},[\cdot, \cdot], \delta_{\mathfrak{g}})$ be a Lie bialgebra of codimension one in a vector space $E$, and $V$ a complement of $\mathfrak{g}$ in $E$. Then we have
$BExtd(E,\mathfrak{g})\cong \mathcal{HBI}_{\mathfrak{g}}^2(V,\mathfrak{g})\cong \mathcal{FLB}(\mathfrak{g},\delta_{\mathfrak{g}})/\equiv$.
\end{theo}

\begin{proof}
Set $V=kx$. Since in Lemma \ref{l4}, $p: V\rightarrow \mathfrak{g}$ is a linear map and
$q:V\rightarrow V$ is a linear isomorphism, we can set $p(x)=u\in \mathfrak{g}$ and
$q(x)=\beta x$ for some $\beta\in k^{*}$. Then this theorem follows from Lemma \ref{l4}, Theorem \ref{tt1} and Proposition \ref{ppr1}.
\end{proof}

\begin{rema}
Let $(\alpha, D^{'}, A, B^{'}) \in \mathcal{FLB}(\mathfrak{g},\delta_{\mathfrak{g}})$. By Lemma \ref{l4} and Theorem \ref{th6}, $(\alpha, D, A, B)$ also belongs to $\mathcal{FLB}(\mathfrak{g},\delta_{\mathfrak{g}})$ where $D$ and $B$ are obtained from $D^{'}$ and $B^{'}$ by (\ref{eqq1}) and (\ref{eqq2}).
\end{rema}

\begin{coro}\label{cc1}
Let $(\mathfrak{g},[\cdot, \cdot], \delta_{\mathfrak{g}})$ be a Lie bialgebra of codimension one in a vector space $E$, and $V$ a complement of $\mathfrak{g}$ in $E$. Suppose that $[\mathfrak{g},\mathfrak{g}]=\mathfrak{g}$ and $Z(\mathfrak{g})=0$. Then $BExtd(E,\mathfrak{g})\cong \mathcal{TFLB}(\mathfrak{g},\delta_{\mathfrak{g}})/\equiv$, where $(0,D,0, B)\equiv(0,D^{'},0, B^{'})$ if and only if there exist
some element $U\in \mathfrak{g}$ and $\beta\in k^{*}$ such that
\begin{eqnarray}
&&D(a)=[U,a]+\beta D^{'}(a),\\
&&B=\delta_{\mathfrak{g}}(U)+\beta B^{'}.
\end{eqnarray}
\end{coro}
\begin{proof}
Since $[\mathfrak{g},\mathfrak{g}]=\mathfrak{g}$ and $Z(\mathfrak{g})=0$, we get $\alpha=0$ and $A=0$ by (\ref{qe1}). Then it can be directly obtained by
Theorem \ref{th6}.
\end{proof}
\begin{coro}\label{cof1}
Let $(\mathfrak{g},[\cdot, \cdot], \delta_{\mathfrak{g}})$ be a Lie bialgebra of codimension one in a vector space $E$, and $V$ a complement of $\mathfrak{g}$ in $E$. Suppose that $[\mathfrak{g},\mathfrak{g}]=\mathfrak{g}$, $Z(\mathfrak{g})=0$, $Der(\mathfrak{g})=Inn(\mathfrak{g})$ and $\wedge^2(\mathfrak{g})^\mathfrak{g}=0$. Then $BExtd(E,\mathfrak{g})=0$.
\end{coro}
\begin{proof}
By Corollary \ref{cc1}, $\alpha=0$ and $A=0$. Since $Der(\mathfrak{g})=Inn(\mathfrak{g})$, we can set $D(a)=[b,a]$ for some $b\in \mathfrak{g}$. Then taking them into (\ref{qe2}), we can get $a.B=0$ for all $a\in \mathfrak{g}$. According to $\wedge^2(\mathfrak{g})^\mathfrak{g}=0$, one can get
$B=0$. Therefore, this result follows by Corollary \ref{cc1}.
\end{proof}
\begin{coro}
Let $(\mathfrak{g},[\cdot, \cdot], \delta_{\mathfrak{g}})$ be a Lie bialgebra of codimension one in a vector space $E$, and $V$ a complement of $\mathfrak{g}$ in $E$. Suppose that $\mathfrak{g}$ is a finite-dimensional semisimple Lie algebra. Then $BExtd(E,\mathfrak{g})=0$.
\end{coro}
\begin{proof}
Since $\mathfrak{g}$ is a finite-dimensional semisimple Lie algebra, $[\mathfrak{g},\mathfrak{g}]=\mathfrak{g}$, $Z(\mathfrak{g})=0$, $Der(\mathfrak{g})=Inn(\mathfrak{g})$ and $\wedge^2(\mathfrak{g})^\mathfrak{g}=0$. Then this result follows by Corollary \ref{cof1}.
\end{proof}

Finally, we present a concrete example to compute $\mathcal{HBI}_{\mathfrak{g}}^2(V,\mathfrak{g})$ when $\text{dim} V=1$.
\begin{exam}
Let $\mathfrak{g}=\mathbb{C}h\oplus\mathbb{C}x \oplus \mathbb{C}y$ be the three-dimensional complex Heisenberg Lie algebra, i.e. the Lie bracket is given by
\begin{eqnarray*}
[h, x]=[h,y]=0,\;\;[x,y]=h.
\end{eqnarray*}
There is a Lie bialgebra structure on $\mathfrak{g}$ with the Lie co-bracket $\delta_\mathfrak{g}$ given by
\begin{eqnarray*}
\delta_\mathfrak{g}(h)=0,~~~\delta_\mathfrak{g}(x)=y\wedge h,~~~~\delta_\mathfrak{g}(y)=h\wedge x.\end{eqnarray*}

Since $\delta_\mathfrak{g}(A)=0$, $A=kh$ for some $k\in \mathbb{C}$. Obviously, the center of $\mathfrak{g}$ is $Z(\mathfrak{g})=\mathbb{C}h$. By (\ref{qe1}),
 we get
\begin{eqnarray}\label{eq31}
0=[a, A]=[a, kh]=\sum\alpha(a_{(2)})a_{(1)}.
\end{eqnarray}
Since $[\mathfrak{g},\mathfrak{g}]=\mathbb{C}h$, $\alpha(h)=0$ by (\ref{qe1}). Then it can be directly obtained from (\ref{eq31}) that $\alpha(y)=\alpha(x)=0$. Therefore, $\alpha=0$. Consequently,
$D$ is a derivation of $\mathfrak{g}$. By some computations, we obtain that $D$ is of the form
\begin{eqnarray*}
D(x, y, h)=(x, y, h)\left(
                      \begin{array}{ccc}
                        a_1 & b_1 & 0 \\
                        a_2 & b_2 & 0 \\
                        a_3 & b_3 & a_1+b_2 \\
                      \end{array}
                    \right).
\end{eqnarray*}

Let $U=b_3x-a_3y$, $\beta=1$ in (\ref{eqq1}). We get $D\equiv D^{'}$, where
\begin{eqnarray*}
D^{'}(x, y, h)=(x, y, h)\left(
                      \begin{array}{ccc}
                        a_1 & b_1 & 0 \\
                        a_2 & b_2 & 0 \\
                        0 & 0& a_1+b_2 \\
                      \end{array}
                    \right).
\end{eqnarray*}
Replace $D$ by $D^{'}$ in (\ref{qe13}) and (\ref{qe2}). Set $B=e_1x\wedge y+e_2y\wedge h+e_3h\wedge x$. Taking it into (\ref{qe2}) and setting $a=x$, $a=y$ and $a=h$ respectively, we obtain that
\begin{eqnarray}
&&\label{pq1}a_1k+e_1-a_2-b_1=0,~~~~~a_2k-2b_2=0,\\
&& \label{pq2}b_2k+e_1+b_1+a_2=0,~~~~~b_1k+2a_1=0.
\end{eqnarray}
By (\ref{qe13}), we get
\begin{eqnarray}
\label{pq3}ke_1=0.
\end{eqnarray}
Therefore, when $k=0$, we have $e_1=a_1=b_2=0$ and $a_2=-b_1$. In the case when $k\neq 0$, we have
\begin{itemize}
\item $k\in \mathbb{C}^\ast$, $e_1=a_1=a_2=b_1=b_2=0$,
\item $k=\pm 2\sqrt{-1}$, $e_1=0$, $a_1=\mp \sqrt{-1}b_1$, $a_2=b_1$ and $b_2= \pm\sqrt{-1}b_1$, where $b_1\in \mathbb{C}^\ast$.
\end{itemize}
Consequently, by Theorem \ref{th6},
any element in $\mathcal{FLB}(\mathfrak{g},\delta_{\mathfrak{g}})$ is equivalent to $(0, D, A, B)$ which is of the form
\begin{itemize}
\item $A=kh$, $D=0$, $B=e_2y\wedge h+e_3h\wedge x$, where $k\in \mathbb{C}$,
\item $A=0$, $D(x, y, h)=(x, y, h)\left(
                      \begin{array}{ccc}
                       0& b_1 & 0 \\
                        -b_1 & 0 & 0 \\
                        0 & 0& 0 \\
                      \end{array}
                    \right)$, $B=e_2y\wedge h+e_3h\wedge x$, where $b_1\neq 0$,
\item $A=\pm 2\sqrt{-1}h$, $D(x, y, h)=(x, y, h)\left(
                      \begin{array}{ccc}
                       \mp \sqrt{-1}b_1& b_1 & 0 \\
                        b_1 & \pm\sqrt{-1}b_1 & 0 \\
                        0 & 0& 0 \\
                      \end{array}
                    \right)$, $B=e_2y\wedge h+e_3h\wedge x$, where $b_1\neq 0$.
\end{itemize}

Therefore, by Theorem \ref{th6} and the choice of $\beta$ in Definition \ref{defi6}, $\mathcal{HBI}_{\mathfrak{g}}^2(V,\mathfrak{g})$ when $\text{dim} V=1$ can be descried by the disjoint union of the following sets.
\begin{itemize}
\item $(0, 0, kh, e_2y\wedge h+e_3h\wedge x)$, where $(0, 0, kh, e_2y\wedge h+e_3h\wedge x)\equiv(0, 0, k^{'}h, e_2^{'}y\wedge h+e_3^{'}h\wedge x) $ if and only if $k=k^{'}$ and there exists $\beta\in \mathbb{C}^{\ast}$ such that $e_2=\beta e_2^{'}$ and $e_3=\beta e_3^{'}$.
\item $(0, D, 0, e_2y\wedge h+e_3h\wedge x)$, where $D(x, y, h)=(x, y, h)\left(
                      \begin{array}{ccc}
                       0& 1 & 0 \\
                        -1 & 0 & 0 \\
                        0 & 0& 0 \\
                      \end{array}
                    \right)$, and $e_2$, $e_3\in \mathbb{C}$. In this case, $(0, D, 0, e_2y\wedge h+e_3h\wedge x)\equiv (0, D, 0, e_2^{'}y\wedge h+e_3^{'}h\wedge x)$ if and only if $e_2=e_2^{'}$ and $e_3=e_3^{'}$.
\item $(0, D, \pm 2\sqrt{-1}h,e_2y\wedge h+e_3h\wedge x)$, where
$D(x, y, h)=(x, y, h)\left(
                      \begin{array}{ccc}
                       \mp \sqrt{-1}& 1 & 0 \\
                        1 & \pm\sqrt{-1} & 0 \\
                        0 & 0& 0 \\
                      \end{array} \right)$, and $e_2$, $e_3\in \mathbb{C}$. In this case, $(0, D, \pm 2\sqrt{-1}h, e_2y\wedge h+e_3h\wedge x)\equiv (0, D, \pm 2\sqrt{-1}h, e_2^{'}y\wedge h+e_3^{'}h\wedge x)$ if and only if $e_2=e_2^{'}$ and $e_3=e_3^{'}$.
\end{itemize}

\end{exam}

\section*{Acknowledgements}
This work was supported by the Zhejiang Provincial Natural Science Foundation of China (No. LY20A010022), the National Natural Science Foundation of China (No. 11871421) and the Scientific Research Foundation of Hangzhou Normal University (No. 2019QDL012).

\bigskip

Yanyong Hong
\vspace{2pt}

School of Mathematics, Hangzhou Normal University
Hangzhou 311121,  China

\vspace{2pt}
hongyanyong2008@yahoo.com

\bigskip


\small

\begin{thebibliography}{9999}\vskip0pt\small
\def\re{\bibitem}\parindent=2ex\parskip=-2pt\baselineskip=-2pt
\bibitem{AM1} A. Agore, G. Militaru, Extending structures for Lie algebras, Monatsh. Math., {\bf 174} (2014), 169-193.
\bibitem{AM2}  A. Agore, G. Militaru,  Extending structures II: The quantum version, J. Algebra, {\bf 336} (2011), 321-341.
\bibitem{ARK} D. Abedi-Fardad, A. Rezai-Agdam,G. Khagigatdust, Classification of four-dimensional real Lie bialgebras of symplectic type and their Poisson-Lie groups. (Russian) Teoret. Mat. Fiz. {\bf 190} (2017),  3-20; translation in Theoret. and Math. Phys., {\bf 190} (2017), 1-17.
\bibitem{B}M. Benayed, Matched pairs and extensions of Lie bialgebras, Extracta Math., {\bf 13} (1998), 255-261.
\bibitem{B1}M. Benayed, Central extensions of Lie bialgebras and Poisson-Lie groups, J. Geom. Phys., {\bf 16} (1995), 301-304.
\bibitem{CP}V. Chari, A. Pressley, A guide to quantum groups, Cambridge University Press, Cambridge (1994).
\bibitem{Dr1} V. Drinf'eld, Hamiltonian structures on Lie groups, Lie bialgebras and the geometric meaning of the classical Yang-Baxter equation, Sov. Math. Dokl., {\bf 27}
(1983), 68-71.
\bibitem{FJ} M. Farinati, A. P.Jancsa, Trivial central extensions of Lie bialgebras, J. Algebra, {\bf 390} (2013), 56-76.
\bibitem{FJ1}M. Farinati, A. P. Jancsa, Three dimensional real Lie bialgebras, Rev. Un. Mat. Argentina, {\bf 56} (2015), 27-62.
\bibitem{G}X. Gomez, Classification of three-dimensional Lie bialgebras. J. Math. Phys., {\bf 41} (2000), 4939-4956.
\bibitem{J}Q. Jin, A classification of low dimensional Lie bialgebras. Comm. Algebra, {\bf 31} (2003), 11, 5481-5499.
\bibitem{K}C. Klim$\check{c}$$\acute{i}$k, Poisson-Lie $T$-duality. Nucl. Phys. B (Proc. Suppl.), {\bf 46} (1996), 116-121.
\bibitem{KS} C.Klim$\check{c}$$\acute{i}$k, P. $\check{S}$evera, Dual non-Abelian duality and the Drinfeld double, Phys. Lett. B, {\bf 351} (1995), 455-462.
\bibitem{Mj}S. Majid, Foundations of quantum group theory, Cambridge Univ. Press, Cambridge, 1995.
\bibitem{Ma}A. Masuoka, Extensions of Hopf algebras and Lie bialgebras, Trans. AMS., {\bf 352} (2000), 3837-3879.
\bibitem{Mi}W. Michaelis, Lie coalgebras, Adv. Math., {\bf 38} (1980), 1-54.
\bibitem{NT}S. Ng, E. Taft, Classification of the Lie bialgebra structures on the Witt and Virasoro algebras, J. Pure Appl. Algebra, {\bf 151} (2000), 67-88.
\bibitem{RHR}A. Rezaei-Aghdam, M., Hemmati, A. R. Rastkar,  Classification of real three-dimensional Lie bialgebras and their Poisson-Lie groups. J. Phys. A, {\bf 38} (2005), 3981-3994.
\bibitem{T}E. Taft, Witt and Virasoro algebras as Lie bialgebras. J. Pure Appl. Algebra, {\bf 87} (1993), 301-312.
\bibitem{Z1}S. Zhang, Classical Yang-Baxter equation and low-dimensional triangular Lie bialgebras over arbitrary field, Int. Math. J., {\bf 2} (2002), 825-831.
\bibitem{Z2}S. Zhang, Classical Yang-Baxter equation and low-dimensional triangular Lie bialgebras, Phys. Lett. A, {\bf 246} (1998), 71-81.


\end{thebibliography}
\end{document}